\documentclass[a4paper,11pt]{amsart}
\usepackage[utf8]{inputenc}
\usepackage{amsmath,amssymb,anysize,enumerate}
\usepackage[cmtip,all]{xy}
\usepackage{hyperref}
\usepackage{cite}
\marginsize{2.4cm}{2.4cm}{2.4cm}{2.4cm}

\renewcommand{\AA}{\mathbb{A}}
\newcommand{\CC}{\mathbb{C}}
\newcommand{\gm}{{\mathbb G}_m}

\newcommand{\Z}{\mathbb{Z}}
\newcommand{\kk}{\Bbbk}
\newcommand{\LL}{\mathbf L}
\newcommand{\RR}{\mathbf R}
\newcommand{\DD}{\mathcal D}
\newcommand{\HH}{\mathcal H}

\newcommand{\KK}{\mathcal K}
\newcommand{\MM}{\mathcal M}

\newcommand{\OO}{\mathcal O}
\newcommand{\xx}{\underline{x}}
\newcommand{\fii}{\varphi}
\newcommand{\fia}{\fii_{\alpha}}
\newcommand{\fiai}{\fia^{-1}}

\newcommand{\bb}{\bar{b}}
\newcommand{\dd}{\partial}
\newcommand{\D}{\text{D}^\text{b}}
\newcommand{\Hom}{\operatorname{Hom}}
\newcommand{\Ext}{\operatorname{Ext}}
\newcommand{\DR}{\operatorname{DR}}
\newcommand{\Spec}{\operatorname{Spec}}
\newcommand{\id}{\operatorname{id}}
\newcommand{\supp}{\operatorname{supp}}
\newcommand{\ra}{\rightarrow}
\newcommand{\lra}{\longrightarrow}
{\theoremstyle{definition}
\newtheorem{defi}{Definition}[section]}
{\theoremstyle{remark}
\newtheorem{nota}[defi]{Remark}}
\newtheorem{teo}[defi]{Theorem}
\newtheorem{prop}[defi]{Proposition}
\newtheorem{lema}[defi]{Lemma}
\newtheorem{coro}[defi]{Corollary}

\begin{document}

\title{Exponents of some one-dimensional Gauss-Manin systems}
\author{Alberto Castaño Domínguez}
\thanks{The author is partially supported by FQM218, P12-FQM-2696, MTM2013-46231-P, ERDF and the project SISYPH: ANR-13-IS01-0001-01/02 and DFG grants HE 2287/4-1 and SE 1114/5-1.}
\email{alcas@hrz.tu-chemnitz.de}
\address{Fakultät für Mathematik, Technische Universität Chemnitz. Reichenhainer Str. 39 09126 Chemnitz (Germany)}

\begin{abstract}
In this paper we provide a purely algebraic characterization of the exponents of one-dimensio\-nal direct images of a structure sheaf by a rational function, related to the vanishing of the cohomologies of a certain Koszul complex associated with such a morphism. This can be extended to a more general family of Gauss-Manin systems. As an application, we calculate a set of possible exponents of the Gauss-Manin cohomology of some arrangements of hyperplanes with multiplicities, relevant to Dwork families and mirror symmetry.
\end{abstract}

\maketitle

\section{Introduction}
Let $\kk$ be an algebraically closed field of characteristic zero. An algebraic variety, or just variety, will mean for us an equidimensional quasi-projective separated finite type scheme over $\kk$, reducible or not. For any smooth variety $\mathcal{X}$, $\D(\DD_{\mathcal{X}})$ will denote the category of bounded complexes of $\DD_{\mathcal{X}}$-modules.

For an open subvariety of the affine line, the exponents of a $\DD$-module over it are strongly related to the monodromy of its solutions. This notion is topological in nature when $\kk=\CC$, but we can manage to work in an algebraic way with a similar concept, and because of that, we will usually use both names, monodromy and exponents, to denote the phenomenon and the object of study. Although this theory can be constructed in any dimension thanks to the formalism of the $V$-filtration, the Bernstein-Sato polynomial and the vanishing cycles of Malgrange and Kashiwara (cf. \cite{MalBernstein,KasVanishing,MaisMeb} or the appendix by Mebkhout and Sabbah at \cite[Ch. III, \S4]{Meb6oper}), it is defined in a much more simple way in dimension one. In fact, we will follow the approach of \cite[\S2.11]{KaESDE}.

The main aim of this paper is to prove the following result:

\begin{teo}\label{Koszul}
Let $n$ be a fixed positive integer, and let $g\in\kk[x_1,\ldots,x_n]$ be a nonzero polynomial. Now let $R=\kk((t))[x_1,\ldots,x_n,g^{-1}]$, $f\in\kk[x_1,\ldots,x_n,g^{-1}]$, and denote by $f'_i$ the partial derivatives of $f$ with respect to the variables $x_i, i=1,\ldots,n$, and by $G$ the closed subvariety $\{g(\xx)=0\}\subseteq\AA^n$. Let $\alpha\in\kk$ and let $\fia$ be the endomorphism of $\kk((t))$ given by $\dd_t-\alpha t^{-1}$. Denote also by $f$ the associated morphism $\AA^n-G\ra \AA^1$. Then, $\alpha\!\mod\Z$ is not an exponent at the origin of any of the $\DD_{\AA^1}$-modules $\HH^if_+\OO_{\AA^n-G}$ if and only if the morphism
$$\begin{array}{rrcl}
\Phi:&R^{n+1}&\lra&R,\\
&(a,b^1,\ldots,b^n)&\longmapsto&(f-t)a+(\dd_1+f'_1\fia)b^1+\cdots+(\dd_n+f'_n\fia)b^n\end{array}$$
is surjective. If it is not surjective, the number of Jordan blocks associated with $\alpha$ of the zeroth cohomology is the dimension of the cokernel of $\Phi$ as a $\kk$-vector space.
\end{teo}

As the reader can check, the statement and its proof are completely independent of the choice of $\kk$, providing a purely algebraic way of dealing with the exponents of a morphism for any such field, apart from their definition itself.

As indicated in the abstract, we will apply this result to a more general context, regarding Gauss-Manin systems over an open subset of the affine line. However, the notion of exponent is really local in nature, so that an open subvariety actually plays an irrelevant role and in the following we will consider everything over the whole affine line.

In the final section we provide two examples illustrating the usefulness of the result. More concretely, we start with a well-known fact about quasi-homogeneous singularities and then we give a result about the exponents of a Gauss-Manin system associated to a special family of arrangements of hyperplanes, which is our second main result:

\begin{teo}\label{exponentes}
Let $(w_0,\ldots,w_n)\in\Z_{>0}^{n+1}$ be an $(n+1)$-tuple of positive integers, and let $\lambda=x_1^{w_1}\cdot\ldots\cdot x_n^{w_n}(1-x_1-\cdots-x_n)^{w_0}$. Then $\alpha\in\kk$ is an exponent at the origin of some cohomology of $\lambda_+\OO_{\AA^n}$ only if $w_i\alpha$ is an integer for some $i=0,\ldots,n$. Moreover, the multiplicity of every exponent of the form $j/w_i$ is the same, without counting coincidences among some $j/w_i$ for different values of $i=0,\ldots,n$ and $j=1,\ldots,w_i$.
\end{teo}

This calculation does not belong only to the realm of hyperplane arrangements, but appears in other interesting contexts. Namely, it can appear when we study the Gauss-Manin cohomology of a generalized Dwork family (cf. \cite[\S3]{CaDwork} and \cite{KaAnother} to know more, respectively, about that relation or about Dwork families in general). That was in fact the main motivation for overcoming this problem. In addition, in \cite[\S3]{CaDwork} we explain how the Gauss-Manin system associated with the morphism $\lambda$ is strongly related to the restriction of a linear form to a torus in $\AA^{n+1}$ (in fact the latter is just an inverse image by an \'{e}tale covering of the former), a setting already treated in \cite{DouSab}, for instance, and of importance in mirror symmetry, for it gives a description of the quantum cohomology of a weighted projective space.

\textbf{Acknowledgements.} This work evolved from a part of the doctoral thesis of the author, advised by Luis Narv\'{a}ez Macarro and Antonio Rojas Le\'{o}n. The author wants to thank the former for suggesting the idea that motivated the main result to him, and both of them for their support and encouragement. He also want to thank them and Christian Sevenheck for their useful comments and discussions about this work, Jean-Baptiste Teyssier for his interesting questions regarding the main result and the anonymous referees for their many useful comments to improve this text.

\section{Preliminaries}

In this section we will recall the basic concepts from $\DD$-module theory that we will need in the following.

\begin{defi}
Let $f:X\ra Y$ be a morphism of smooth varieties. The \emph{direct image} of complexes of $\DD_X$-modules is the functor $f_+:\D(\DD_X)\ra \D(\DD_Y)$ given by
$$f_+\MM:=\RR f_*\left(\DD_{Y\leftarrow X}\otimes_{\DD_X}^{\LL}\MM\right),$$
where $\DD_{Y\leftarrow X}$ is the $\left(f^{-1}\DD_Y,\DD_X\right)$-bimodule
$$\DD_{Y\leftarrow X}:=\omega_X\otimes_{f^{-1}\OO_Y}f^{-1}\Hom_{\OO_{Y}} \left(\omega_Y,\DD_Y\right).$$
The \emph{inverse image} of complexes of $\DD_Y$-modules is the functor $f^+:\D(\DD_Y)\ra \D(\DD_X)$ given by
$$f^+\MM:=\DD_{X\ra Y}\otimes_{f^{-1}\DD_Y}^{\LL}f^{-1}\MM,$$
where $\DD_{X\ra Y}$ is the $\left(\DD_X,f^{-1}\DD_Y\right)$-bimodule
$$\DD_{X\ra Y}:=\OO_X\otimes_{f^{-1}\OO_Y} f^{-1}\DD_Y.$$
\end{defi}

\begin{nota}
When $f:X=Y\times Z\ra Z$ is a projection, $\DD_{Z\leftarrow X}\otimes_{\DD_X}^{\LL}\MM$ is nothing but a shifting by $\dim Y$ places to the left of the relative de Rham complex of $\MM$
$$\DR_f(\MM):=0\lra\MM\lra\MM\otimes_{\OO_{X}}\Omega_{X/Z}^1\lra\ldots\lra \MM\otimes_{\OO_{X}}\Omega_{X/Z}^n\lra0,$$
so we will have $f_+\cong\RR f_*\DR_f(\bullet)[\dim Y]$ (\hspace{-.5pt}\cite[Ch. I, Lem. 5.2.2]{Meb6oper}).
\end{nota}

We will be interested in the case in which $X$ is an open subvariety of the affine line. From now on, we will denote by $D_x$ the product $x\dd_x$, omitting the variable as long as it is clear from the context.

\begin{defi}
A \emph{Kummer $\DD$-module} is the quotient $\KK_\alpha=\DD_{\gm}/(D-\alpha)$, for any $\alpha\in\kk$.
\end{defi}

\begin{nota}
Note that any two Kummer $\DD$-modules $\KK_\alpha$ and $\KK_\beta$ are isomorphic if and only if $\alpha-\beta$ is an integer. Then $\KK_\alpha\cong\OO_{\gm}$ for any $\alpha\in\Z$.
\end{nota}

\begin{prop}\label{regirreg}
Let $\MM$ be a holonomic $\DD_X$-module, let $p$ be a point of $X$, and fix a formal parameter $x$ at $p$ such that $\widehat{\OO}_{X,p}\cong\kk[[x]]$. The tensor product $\MM\otimes_{\OO_X}\kk((x))$ can be decomposed as the direct sum of its regular and purely irregular parts.

Now assume that $\MM\otimes_{\OO_X}\kk((x))\cong\kk((x))[D]/(L)$, where
$$L=\sum_ix^iA_i(D)\in\kk[[x]][D],$$
with $\deg_DL=g\geq g_0=\deg_DA_0$ (so that $A_0\neq0$). As a consequence, the rank of $\left(\MM\otimes_{\OO_X}\kk((x))\right)_{\emph{reg}}$ is $g_0$, and if this last degree is positive and $A_0(t)=\gamma\prod_i(t-\alpha_i)^{n_i}$, its composition factors are $\KK_{\alpha_i,p}$ with multiplicity $n_i$, where $\KK_{\beta,p}$ is the tensor product $\kk((x))\otimes_{\kk[x^{\pm}]}\KK_{\beta}\cong\kk((x))[D]/(D-\beta)$.

Moreover, if the roots of $A_0(t)$ are not congruent modulo $\Z$, then
$$\left(\MM\otimes_{\OO_X}\kk((x))\right)_{\emph{reg}}\cong \kk((x))[D]/(A_0(D))\cong \bigoplus_i\kk((x))[D]/ (D-\alpha_i)^{n_i}.$$
\end{prop}
\begin{proof}
The decomposition into regular and purely irregular parts of the tensor product $\MM\otimes_{\OO_X}\kk((x))$ is a well-known fact of the theory of integrable connections (cf. \cite[Ch. III, Thm. 1.5, Cor. 1.7]{MalEqDiff}).

The rest is analogous to \cite[Cor. 2.11.7]{KaESDE}. Although that result and those on which it depends at \cite[\S2.11]{KaESDE} are stated over $\CC$, their proofs are purely algebraic and, in fact, can be generalized for any algebraically closed field of characteristic zero.
\end{proof}

\begin{prop}[Formal Jordan decomposition lemma]\label{Jordan}
Let $\MM$, $p$ and $x$ be as before, and suppose that $\MM$ is regular at $p$. Then,
\begin{enumerate}
\item[(i)] $\MM\otimes_{\OO_X}\kk((x))$ is the direct sum of regular indecomposable $\kk((x))[D]$-modules;
\item[(ii)] writing
    $$\operatorname{Loc}(\alpha,n_\alpha):=\kk((x))[D]/ (D-\alpha)^{n_\alpha},$$
    then, for any two $\kk((x))[D]$-modules $\operatorname{Loc}(\alpha,n_\alpha)$ and $\operatorname{Loc}(\beta,n_\beta)$, and $i=0,1$, the vector space $\Ext_{\DD_X}^i(\operatorname{Loc}(\alpha,n_\alpha), \operatorname{Loc}(\beta,n_\beta))$ has dimension $\min(n_\alpha,n_\beta)$ if $\alpha-\beta\in\Z$ and zero otherwise;
\item[(iii)] any regular indecomposable $\kk((x))[D]$-module is isomorphic to
    $\operatorname{Loc}(\alpha,n_\alpha)$, where $\alpha$ is unique modulo the integers;
\item[(iv)] given $\alpha\in\kk$, the number of indecomposables of type $\operatorname{Loc}(\alpha,m)$ in the decomposition of $\MM\otimes_{\OO_X}\kk((x))$ is $\dim_\kk\Hom_{\DD_X}\left(\MM,\operatorname{Loc}(\alpha,1)\right)$.
\end{enumerate}
\end{prop}
\begin{proof}
When $\kk=\CC$, there is a topological proof as in \cite[Lem. 2.11.8]{KaESDE}. However, we can give a purely (linear) algebraic one.

Since $\MM$ is holonomic, it is a finitely generated torsion $\DD_X$-module, and so will $\MM\otimes\kk((x))$ be over $\kk((x))[D]$. This ring is a noncommutative principal ideal domain, so by the structure theorem for finitely generated modules over such a ring (cf. \cite[Ch. 3, Thm. 19]{Jacobson}) we obtain that $\MM\otimes\kk((x))$ is the direct sum of indecomposable $\kk((x))[D]$-modules. They must be regular since $\MM$ is, and that proves point (i).

Now let $\operatorname{Loc}(\alpha,n_\alpha)$ and $\operatorname{Loc}(\beta,n_\beta)$ be as in point (ii). We can suppose that both $\alpha$ and $\beta$ belong to the same fundamental domain (exhaustive set of representatives without repetitions) of $\kk/\Z$, up to isomorphism. Since $\operatorname{Loc}(\alpha,n_\alpha)$ is a flat $\kk((x))$-module, we can assume that $\alpha=0$. Now note that the vector spaces $\Ext_{\DD_X}^i(\operatorname{Loc}(\alpha,n_\alpha), \operatorname{Loc}(\beta,n_\beta))$ are just the kernel and the cokernel of $D^{n_\alpha}$ over $\operatorname{Loc}(\beta,n_\beta)$. In general, each of the $\operatorname{Loc}(\gamma,m)$, for any $\gamma\in\kk/\Z$ and any $m>0$, is a successive extension of $\operatorname{Loc}(\gamma,1)$. Thus if $\beta\neq0$, since the operator $D$ is clearly bijective on them, both $\Ext$ spaces vanish. If $\beta=0$, then the statement is easy to check.

Let us go now for point (iii). Thanks to the discussion of the first point, we can affirm that
$$\MM\otimes\kk((x))\cong\bigoplus_{i=1}^r\kk((x))[D]/(A_i(x,D)),$$
where $A_i(x,D)=\sum_{j\geq0}x^jA_{ij}(D)$. By Proposition \ref{regirreg},
$\kk((x))[D]/(A_i(x,D))$ is isomorphic to a successive extension of the $\operatorname{Loc}(\alpha_i,1)$, the $\alpha_i$ being the roots of each $A_{i0}$. Now we just need to invoke the previous point; such an extension must be a direct sum of some $\operatorname{Loc}(\beta,n_\beta)$, with, possibly, some $n_\beta>1$ if the roots of $A_{i0}$ were congruent modulo the integers.

Point (iv) is just an easy consequence of the two preceding ones.
\end{proof}

These two propositions show that the equivalence classes modulo $\Z$ of the numbers $\alpha$ appearing in the decomposition of the tensor product of a holonomic $\DD_X$-module with $\kk((x))$, and their associated $n_\alpha$, are intrinsic to the $\DD_X$-module and it is quite important, actually, to know its behaviour at a point, so that motivates the following definition.

\begin{defi}
Let $\MM$, $p$ and $x$ be as in Proposition \ref{regirreg}. The \emph{exponents} of $\MM$ at $p$ are the values $\alpha_i\in\kk$ such that
$$\left(\MM\otimes_{\OO_X}\kk((x))\right)_{\text{reg}}\cong \bigoplus_i\operatorname{Loc}(\alpha_i,n_i),$$
seen as elements of $\kk/\Z$. For each exponent $\alpha_i$ we define its \emph{multiplicity} as $n_i$.
\end{defi}

\begin{nota}
For the sake of simplicity, we will usually denote both an exponent and some of its representatives in $\kk$ in the same way.

Exponents are considered unordered and possibly repeated. Note that, when $\kk=\CC$, this notion of multiplicity of an exponent $\alpha$ is related to the size of the Jordan blocks of the local monodromy associated with the eigenvalue $e^{2\pi i\alpha}$, and not to its multiplicity as a root of the characteristic polynomial of the monodromy. However, these two notions are the same under some special conditions (cf. \cite[Cor. 3.2.2, Lem. 3.7.2]{KaESDE}). Nevertheless, in our algebraic setting, whenever we mention ``Jordan block'' we will mean a regular indecomposable $\operatorname{Loc}(\alpha,n_\alpha)$, in analogy with the complex analytic case.

Although, as we have seen, the exponents at the origin of a $\DD_{\AA^1}$-module can be defined even if it has an irregular singularity there, in the following we will deal with complexes of regular holonomic $\DD_{\AA^1}$-modules, since direct image preserves regularity (cf. \cite[Ch. II, Thm. 9.3.1]{Meb6oper}).
\end{nota}

\section{Gauss-Manin systems, main result and Laurent series}\label{centro}

Let us recall now the basic setting of one-dimensional Gauss-Manin systems, seen from the point of view of $\DD$-module theory.

Fix a positive integer $n$, some variables $x_1,\ldots,x_n$ and a special one called $\lambda$. Consider an open set $U\subseteq\AA^n=\Spec\left(\kk[x_1,\ldots,x_n]\right)$ and a smooth variety $X\subset U\times \AA^1=U \times\Spec(\kk[\lambda])$, together with the second projection $\pi_2:X\ra \AA^1$. In terms of $\DD$-modules, the Gauss-Manin cohomology, or system, of $X$, is just the direct image of the structure sheaf $\pi_{2,+}\OO_{X}$. It is a complex of $\DD_{\AA^1}$-modules, so we could be interested in knowing its behaviour at the origin, and in particular its exponents. In this paper we will focus on the case where $X$ is a hypersurface.

Going back to Theorem \ref{Koszul}, the direct image $f_+\OO_{\AA^n-G}$ can be seen as the Gauss-Manin cohomology of the graph of $f$ in $(\AA^n-G)\times \AA^1$. However, that is a rather concrete and simple example of a family of hypersurfaces. We will explain after proving our main result how to relate this setting to a broader family of Gauss-Manin systems.

\noindent\textit{Proof of Theorem \ref{Koszul}.} Let us first deal with the zeroth cohomology of $f_+\OO_{\AA^n-G}$. After finishing with it we will justify the extension of the statement to all of them.

Let $K=f_+\OO_{\AA^n-G}$. By Proposition \ref{Jordan} we can claim that
$$\HH^0(K)\otimes_{\OO_{\AA^1}}\kk((t))\cong \bigoplus_{i=1}^r \kk((t))[D]/(D-\beta_i)^{m_i}.$$
From the second paragraph of the same proposition we can deduce that $\alpha$ will not be an exponent of $\HH^0(K)$ if and only if the endomorphism
$$D-\alpha\cdot:\HH^0(K)\otimes_{\OO_{\AA^1}}\kk((t)) \lra\HH^0(K)\otimes_{\OO_{\AA^1}}\kk((t))$$
is bijective, or even surjective, for the
$$\operatorname{Ext}_{\kk((t))[D]}^i \left(\kk((t))[D]/(D-\alpha),\kk((t))[D]/(D-\beta)^k\right)$$
do not vanish or do vanish at the same time, whenever $\alpha$ is or is not congruent to $\beta$ modulo $\Z$, respectively, for $i=0,1$ and any $k$.

Now let us decompose the morphism $f$ as the closed immersion into its graph $i_{\Gamma}$ followed by the projection $\pi$ on the first coordinates. By Kashiwara's equivalence and the properties of local cohomology (cf. \cite[Thm. 1.6.1, Prop. 1.7.1]{HTT}), the complex $i_{\Gamma,+}\OO_{\AA^n-G}$ is concentrated in degree zero and, furthermore, is $\OO_{(\AA^n-G)\times \AA^1}(*\Gamma) /\OO_{(\AA^n-G)\times \AA^1}$, for the graph of $f$ is smooth in $\AA^{n+1}$. Therefore, the fact that $\alpha$ is not an exponent of $\HH^0(K)$ is equivalent to the surjectivity of $D-\alpha$ on
$$\pi_+\OO_{(\AA^n-G)\times \AA^1}(*\Gamma) /\OO_{(\AA^n-G)\times \AA^1}\otimes_{\OO_{\AA^1}}\kk((t)).$$

Note that we are always dealing with affine morphisms and quasi-coherent $\OO_{(\AA^n-G)\times \AA^1}$-modules and we are taking tensor products with $\kk((t))$, so it suffices (cf., for example, \cite[Prop. 1.4.4]{HTT}) to work from now on with the global sections of the objects involved in the proof.

Write $M=\kk[\xx,g^{-1},t]\left[(t-f)^{-1}\right]/\kk[\xx,g^{-1},t]$. Recall that we are interested in the top cohomology of $\DR_{\xx}(M)$. Since $\kk((t))$ is flat over $\kk[t]$, tensor products with the former over the latter commute with cohomology, and thus we are going to deal with
$$M_{\text{loc}}:=\kk((t))[\xx,g^{-1}]\left[(t-f)^{-1}\right]/ \kk((t))[\xx,g^{-1}],$$
which is a module over $R=\kk((t))[x_1,\ldots,x_n,g^{-1}]$ and $\widehat{\DD}:=R\langle\dd_t,\dd_1,\ldots,\dd_n\rangle$.

Let us introduce just a bit more of notation that we are going to use. We will write
$$\DD_t:=\kk((t))\langle\dd_t\rangle,\:\DD_{\xx}:=\kk[\xx,g^{-1}]\langle\dd_1,\ldots,\dd_n\rangle\,\text{ and }\,\widehat\DD_{\xx}:=\kk((t))[\xx,g^{-1}]\langle\dd_1,\ldots,\dd_n\rangle.$$

Summing everything up, $\alpha\!\mod\Z$ is not an exponent of the $\DD_{\AA^1}$-module $\HH^0K$ at the origin if and only if
$$\RR^1\Hom_{\DD_t}\left(\DD_t/(D-\alpha),\RR^n\Hom_{\widehat\DD_{\xx}} \left(R,M_{\text{loc}}\right)\right)=0.$$
Note that
$$R\cong\kk((t))\otimes_\kk\kk[\xx,g^{-1}]\cong\kk((t))\otimes_\kk \DD_{\xx}\otimes_{\DD_{\xx}}\kk[\xx,g^{-1}]\cong \widehat\DD_{\xx}\otimes_{\DD_{\xx}}\kk[\xx,g^{-1}],$$
so by extension of scalars,
$$\RR^n\Hom_{\widehat\DD_{\xx}} \left(R,M_{\text{loc}}\right)\cong \RR^n\Hom_{\DD_{\xx}} \left(\kk[\xx,g^{-1}],M_{\text{loc}}\right).$$
Now applying the derived tensor-hom adjunction,
$$\RR^1\Hom_{\DD_t}\left(\DD_t/(D-\alpha),\RR^n\Hom_{\DD_{\xx}} \left(\kk[\xx,g^{-1}],M_{\text{loc}}\right)\right)$$
$$\cong\RR^{n+1}\Hom_{\DD_{\xx}}\left(\DD_t/(D-\alpha) \boxtimes\kk[\xx,g^{-1}] ,M_{\text{loc}} \right)$$
$$\cong\RR^{n+1}\Hom_{\widehat\DD}\left(\widehat\DD/(D-\alpha,\dd_1, \ldots,\dd_n),M_{\text{loc}}\right),$$
the last isomorphism being by extension of scalars again. (Note that the first isomorphism is of $\kk$-vector spaces.)

Now $M_{\text{loc}}$ is a self-dual $\widehat\DD$-module, being the direct image by a closed immersion of the self-dual object $\kk[\xx,g^{-1}]$ (cf. \cite[Ch. I, Cor. 5.3.13]{Meb6oper}), so by duality $\alpha$ is not an exponent of the $\DD_{\AA^1}$-module $\HH^0K$ at the origin if and only if
$$\RR^{n+1}\Hom_{\widehat\DD}\left(M_{\text{loc}}, \widehat\DD/(D+1+\alpha,\dd_1, \ldots,\dd_n)\right)=0.$$
The second $\widehat\DD$-module above is nothing but $R\cdot t^{-1-\alpha}$, where $t^{-1-\alpha}$ should be understood as a symbol. The actions of the partial derivatives are the usual ones in $R$ of $\dd_1,\ldots,\dd_n$, and regarding $\dd_t$,
$$\dd_t\left(a\cdot t^{-1-\alpha}\right)=\dd_t(a)\cdot t^{-1-\alpha}+(-1-\alpha)t^{-1}a\cdot t^{-1-\alpha}.$$
In order to finish all this construction, take into account that the annihilator of the class of $(t-f)^{-1}$ in $M_{\text{loc}}$ is the left ideal $(f-t,\dd_1+f'_1\dd_t,\ldots, \dd_n+f'_n\dd_t)$; indeed, each of its generators make it vanish and the ideal is maximal. Therefore, $M_{\text{loc}}$ can be presented as
$$M_{\text{loc}}\cong\widehat\DD/(f-t,\dd_1+f'_1\dd_t,\ldots, \dd_n+f'_n\dd_t).$$
Recall that the exponents are equivalence classes in $\kk/\Z$, so we could replace $\alpha$ by $\alpha-1$ in this whole procedure, taking another representative of the same exponent without affecting the validity of the proof. Then we can claim that $\alpha$ is not an exponent of the $\DD_{\AA^1}$-module $\HH^0f_+\OO_{\AA^n}$ at the origin if and only if the $\kk$-linear homomorphism $\Phi: R^{n+1}\lra R$ given by
$$\Phi=\left(f-t,\dd_1+f'_1\fia,\ldots,\dd_n+f'_n\fia\right)$$
is surjective.

The statement on the dimension of the cokernel follows easily by reversing the isomorphisms and equivalences and applying point (ii) of Proposition \ref{Jordan}.

Note that the operators $f-t,\dd_1+f'_1\fia,\ldots,\dd_n+f'_n\fia$ commute pairwise, so the Koszul complex $K^\bullet(R;f-t,\dd_1+f'_1\fia,\ldots, \dd_n+f'_n\fia)$ is well defined. Moreover, thanks to the same choice of the representative of $\alpha\!\mod\Z$ as two paragraphs above we can see that its cohomologies are just the vector spaces
$$\RR^k\Hom_{\widehat\DD}\left(M_{\text{loc}}, \widehat\DD/(D+1+\alpha,\dd_1, \ldots,\dd_n)\right),$$
whose duals are in turn an extension of the
\begin{equation}\label{prueba}
\RR^i\Hom_{\DD_t}\left(\DD_t/(D-\alpha),\RR^j\Hom_{\widehat\DD_{\xx}} \left(R,M_{\text{loc}}\right)\right)
\end{equation}
with $j=k$ and $j=k-1$.

Now we claim that the surjectivity of $\Phi$ is equivalent to the vanishing of all the cohomologies of such a Koszul complex. One implication is trivial; the other is \cite[\S~9, Cor. 1]{Bourbaki}. Then, $\Phi$ is surjective if and only if all of the vector spaces (\ref{prueba}) vanish, which, following an argument analogous to the case of the zeroth cohomology $\HH^0K$, is equivalent to the fact that $\alpha$ is not an exponent of any of the cohomologies of $f_+\OO_{\AA^n-G}$. This ends the proof of the theorem.
\hfill$\square$

We might have that some $\alpha$ is an exponent at the origin of some cohomologies of $f_+\OO_{\AA^n-G}$. In that case, reviewing the final argument of the proof, one can still have a partial result when dealing with the vanishing of the cohomologies of the whole Koszul complex. More concretely, we have the following:

\begin{coro}
Under the same conditions as before, were the Koszul complex $K^\bullet(R;f-t,\dd_1+f'_1\fia,\ldots, \dd_n+f'_n\fia)$ acyclic in degrees $d_0$ to $d_1$ (possibly equal to 0 or $n+1$, respectively), then $\alpha\!\mod\Z$ is not an exponent at the origin of any of the cohomologies $\HH^kf_+\OO_{\AA^n-G}$ for $d_0-1\leq k+n\leq d_1$.
\end{coro}
\begin{proof}
As we noted at the end of the proof of the theorem, if $K^\bullet(R;f-t,\dd_1+f'_1\fia,\ldots, \dd_n+f'_n\fia)$ is acyclic in degree $k$, then
$$\RR^{k}\Hom_{\widehat\DD}\left(M_{\text{loc}}, \widehat\DD/(D+1+\alpha,\dd_1, \ldots,\dd_n)\right)=0,$$
whose dual is the extension of the
$$\RR^i\Hom_{\DD_t}\left(\DD_t/(D-\alpha),\RR^j\Hom_{\widehat\DD_{\xx}} \left(R,M_{\text{loc}}\right)\right)$$
with $j=k$ and $j=k-1$. As a consequence, for every $i$ and $j$ with $d_0-1\leq j\leq d_1$ such an object must vanish, and in conclusion, the endomorphism $$D-\alpha:\HH^j(K)\otimes\kk((t))\lra\HH^j(K)\otimes\kk((t))$$
is surjective for $d_0-1\leq j+n\leq d_1$, so $\alpha\!\mod\Z$ is not an exponent at the origin of any of the cohomologies $\HH^jf_+\OO_{\AA^n-G}$ for such values of $j$.
\end{proof}

The following corollary, despite being an easy consequence of both the theorem and the previous corollary, seems to us interesting enough to be written explicitly:

\begin{coro}
Using the same notation as in the theorem, if $\alpha\!\mod\Z$ is an exponent at the origin of some cohomology $\HH^if_+\OO_{\AA^n-G}$ with $i<0$, then it is also an exponent of $\HH^0f_+\OO_{\AA^n-G}$.
\end{coro}
\begin{proof}
If $\alpha\!\mod\Z$ were not an exponent of $\HH^0f_+\OO_{\AA^n-G}$, then its associated morphism $\Phi$ as in the theorem would be surjective, so as in the end of the proof of our main result, by \cite[\S9, Cor. 1]{Bourbaki} we could claim that every cohomology of the Koszul complex $K^\bullet(R;f-t,\dd_1+f'_1\fia,\ldots, \dd_n+f'_n\fia)$ would vanish. Then by the corollary above, $\alpha\!\mod\Z$ would not be an exponent at the origin of any of the other cohomologies of $f_+\OO_{\AA^n-G}$.
\end{proof}

We return now to the context of Gauss-Manin systems in the form of the following statement:

\begin{prop}
Keep the notation as at the beginning of this section. Let $r(\xx)$ be a polynomial of $\kk[\xx]$, and let $U$ be the basic open set $\{r(\xx)\neq0\}$. Assume there exist two regular functions on $U$, $p(\xx)$ and $q(\xx)$, such that $X$ is defined by the equation $p(\xx)-\lambda^dq(\xx)=0$, for certain $d>0$. Write $f=p/q$, $q=\tilde{q}/r^m$ and $G=\{r\tilde{q}=0\}$. Then, the noninteger exponents and the integer ones with multiplicity greater than one of the Gauss-Manin system $\pi_{2,+}\OO_X$ are $d$ times those of $f_+\OO_{\AA^n-G}$.
\end{prop}
\begin{proof}
Let us show how to reduce ourselves to consider $d=1$. Indeed, form the Cartesian diagram
$$\xymatrix{\ar @{} [dr] |{\Box} X \ar[d]_{\pi_2} \ar[r]^{\id\times[d]} & \tilde{X} \ar[d]^{\pi_2}\\
 \AA^1 \ar[r]^{[d]} & \AA^1,}$$
where $[d]$ just means taking the $d$th power of the argument. It is easy to check that if $X$ is smooth, so is $\tilde{X}$. Therefore, by the base change formula \cite[Thm. 1.7.3]{HTT}, $\pi_{2,+}\OO_X\cong[d]^+\pi_{2,+}\OO_{\tilde{X}}$, so as written above, we could find the exponents of the Gauss-Manin cohomology of $X$ by finding those of $\tilde{X}$; the former will just be $d$ times the latter.

Rename, for the sake of clarity, $\tilde{X}$ as $X$, assuming that $d=1$ throughout the rest of the proof. Write both $p$ and $q$ as fractions with the same denominator $\bar{p}/r^N$ and $\bar{q}/r^N$, respectively. Then $X$ is the vanishing locus of $\bar{p}(\xx)-\lambda\bar{q}(\xx)$ in $U\times \AA^1$. Let $Z$ be the hypersurface of $X$ with equation $\bar{q}(\xx)=0$. Then $Z$ is contained in $X$ and is defined by $\{\bar{p}=\bar{q}=0\}$ in the whole of $U$, so it is the product of a subvariety $Z'\subset U$ with $\AA^1$.

Now we can form the excision triangle (cf. \cite[Ch. I, \S6.1]{Meb6oper})
$$\RR\Gamma_{[Z]}\OO_X\lra\OO_X\lra\OO_X(*Z).$$
Let us see what happens when we apply $\pi_{2,+}$ to the triangle. Let $i$ be the closed immersion $X\ra U\times \AA^1$. Thanks to \cite[Ch. I, Prop. 6.4.1]{Meb6oper}, we can affirm that $i_+\RR\Gamma_{[Z]}\OO_X\cong \RR\Gamma_{[Z]}(i_+\OO_X)\cong\RR\Gamma_{[Z]}\RR\Gamma_{[X]}\OO_{U\times \AA^1}[1]$, for $X$ is smooth in $U$. But $\RR\Gamma_{[Z]}\RR\Gamma_{[X]}\OO_{U\times \AA^1}\cong\RR\Gamma_{[Z]}\OO_{U\times \AA^1}$ because $Z$ is contained in $X$ (cf. \cite[Ch. I, Prop. 6.2.4]{Meb6oper} and beware the typo). Now the latter local cohomology module is nothing but $\pi_1^+\RR\Gamma_{[Z']}\OO_U$, where $\pi_1$ is the first projection $U\times \AA^1\ra U$, by \cite[Ch. I, Prop. 6.3.1]{Meb6oper}. Thus abusing a bit of the notation so that $\pi_2$ represents two different projections onto $\AA^1$,
$$\pi_{2,+}\RR\Gamma_{[Z]}\OO_X\cong \pi_{2,+}\pi_1^+\RR\Gamma_{[Z']}\OO_U.$$
And now it is easy to see that this complex has only copies of $\OO_{\AA^1}$ among its cohomologies; it is simply a consequence of applying the base change formula to the Cartesian square
$$\xymatrix{\ar @{} [dr] |{\Box} U\times \AA^1 \ar[d]_{\pi_1} \ar[r]^{\pi_2} & \AA^1 \ar[d]^{\pi_{\AA^1}}\\
 U \ar[r]^{\pi_U} & \{*\},}$$
where $\pi_U$ and $\pi_{\AA^1}$ are the projections from the variety in the subscript to a point. In conclusion, we can claim that $\pi_{2,+}\RR\Gamma_{[Z]}\OO_X$ is just a bunch of copies of the structure sheaf $\OO_{\AA^1}$.

Then apart from the purely constant part, the information about the exponents of $\pi_{2,+}\OO_X$ can be found within $\pi_{2,+}\OO_X(*Z)$. But $X-Z$ can be seen as the graph of $\bar{p}/\bar{q}=p/q$ in $U\times\AA^1$, so now this complex can be realized in the form of our Theorem \ref{Koszul} just by taking $g=r\bar{q}$ and $f=\bar{p}/\bar{q}$.
\end{proof}

\begin{nota}
We have provided in the end a way of computing the noninteger exponents of $\pi_{2,+}\OO_X$. In fact, we could have thought of a slightly broader family of Gauss-Manin systems, namely, those associated with a family of the form $X=\{p(\xx)-\gamma(\lambda)q(\xx)=0\}\subset U\times \AA^1$, for some polynomial $\gamma\in\kk[\lambda]$. The reduction of the beginning of the proof by base change would still be possible, but the calculation of $\gamma^+\pi_{2,+}\OO_{\tilde{X}}$ would not at all be as direct as with $\gamma(\lambda)=\lambda^d$.
\end{nota}

We finish this section by providing several results or notions regarding the field of formal Laurent series that will be of interest later when we tackle a particular example.

\begin{lema}\label{deforgrado}
Let $\fii:\kk((t))\lra\kk((t))$ be a $\kk$-linear automorphism of $\kk((t))$ such that $\fii(\kk[[t]]\cdot t^k)=\kk[[t]]\cdot t^k$ for every $k\in\Z$. Then, for any $\kk$-linear endomorphism $\psi$ of $\kk((t))$ such that $\psi(\kk[[t]]\cdot t^k)\subseteq\kk[[t]]\cdot t^{k+1}$, the sum $\fii+\psi$ is another automorphism of $\kk((t))$.
\end{lema}
\begin{proof}
Multiplying by $\fii^{-1}$ we can assume that $\fii=\id$. We will write the elements of $\kk((t))$ as $a=\sum_ka_kt^k$.

Then let $b$ be a fixed formal Laurent series and let us see if there exists an $a\in\kk((t))$ such that $(\id+\psi)(a)=b$. Evidently, the exponents of the least powers of $t$ (which are called the orders) of both of $a$ and $b$ will be the same, so let us write
$$a=\sum_{k\geq m}a_kt^k\,,\,\psi(a)=\sum_{k\geq m+1}a'_kt^k\,\text{ and }\,b=\sum_{k\geq m}b_kt^k.$$
From the equation $(\id+\psi)(a)=b$ we deduce that $a_m=b_m$. Now call $a^1=a-a_mt^m$ and $b^1=b-(\id+\psi)(a_mt^m)$; both of them have order $m+1$. We have
$$(\id+\psi)a^1=(\id+\psi)a-(\id+\psi)(a_mt^m)=b^1.$$
Thus we can start the same process over again with $a^1$ and $b^1$. Since this can be continued for every power of $t$, we can deduce the surjectivity of $\id+\psi$. Moreover, if we take $b_k=0$ for every $k\in\Z$, it follows that every $a_k$ vanishes too, so $\id+\psi$ is also injective.
\end{proof}

\begin{defi}
Let $r$ be an element of $\kk$. Then we can define the operators $D_{t,r}=t\dd_t+r$ and analogously $D_{i,r}=x_i\dd_i+r$, for $i=1,\ldots,n$. We will write $\fii_r=\dd_t+rt^{-1}=t^{-1}D_{t,r}$ (note the sign change with respect to the previous notation). They are $\kk$-linear endomorphisms of $\kk((t))$, so we can also consider them to be operating within any $\kk((t))$-algebra by extension of scalars.
\end{defi}

\begin{nota}
It is easy to see that $D_{t,r}$ (and so $\fii_r$) is an automorphism of $\kk((t))$ for every $r$ not an integer and only for them, for $D_{t,r}$ sends a power $t^k$ of $t$ to $(k+r)t^k$. In this case we can define another family of operators that will play a fairly main role in the next section:
\end{nota}

\begin{defi}\label{A-operators sencillos}
Fix an element $\alpha$ of $\kk$, and let $r$ and $s$ be two other elements of $\kk$ such that $\alpha+s$ is not an integer. Then we can define the operator $A_{r,s}=t+r\fii_{\alpha+s}^{-1}$.

Let $R_n=\kk((t))[x_1,\ldots,x_n]$ and $\beta\in\kk$. We can also define the $\kk$-linear endomorphisms of $R_n$ given by $A_{\beta D_{i,r},s}=t+\beta D_{i,r}\fii_{\alpha+s}^{-1}$, where $i=1,\ldots,n$.
\end{defi}

In the following, for the sake of simplicity, we will denote by $A_r$, $A_{\beta D_{i,r}}$ and $D_{r}$ the operators $A_{r,0}$, $A_{\beta D_{i,r},0}$ and $D_{t,r}$, respectively.

\begin{nota}\label{Adifiso}
As before, $A_{r,s}$ is not always an automorphism of $\kk((t))$, as $A_{\beta D_{i,r},s}$ is of $R_n$. Since $A_{r,s}\fii_{\alpha+s}=D_{t,\alpha+r+s}$, the former is bijective whenever $\alpha+r+s$ is not an integer. Analogously, $A_{\beta D_{i,r},s}\fii_{\alpha+s}=\beta D_{i,r}+D_{t,\alpha+s}$. It sends $t^k\xx^{\underline{u}}$ to $(\beta(u_i+r)+\alpha+k+s)t^k\xx^{\underline{u}}$, so $A_{\beta D_{i,r},s}$ is bijective if and only if, for every integer $l$, we have that $\beta(l+r)+\alpha+s$ is not an integer.
\end{nota}

Now we could wonder about the commutativity of those operators that we have just defined. We will use the following lemma, whose proof is easy and left to the reader (for each relation, use some of the ones proved before and the Leibniz rule):

\begin{lema}\label{Adifcom}
Let $\alpha$ and $\beta$ be two elements of $\kk$, and $r$, $r'$, $s$ and $s'$ four other elements of $\kk$ such that neither $\alpha+s$ nor $\alpha+s'$ are integers. Then, the following relations hold:
\begin{enumerate}
\item $t\fia=\fii_{\alpha-1}t\,,\:\fia\fii_{\beta} =\fii_{\beta+1}\fii_{\alpha-1}.$
\item $A_{r,s}t=tA_{r,s+1}\,,\:A_{\beta D_{i,r},s}x_i=x_iA_{\beta D_{i,r+1},s}\,,\:A_{r,s}A_{r',s'}=A_{r', s'-1}A_{r,s+1}.$
\end{enumerate}
\end{lema}

\section{Two examples}

As we mentioned in the introduction, we will conclude this note by giving an example of an application of Theorem \ref{Koszul}, focusing on the case in which our morphism $f$ is defined by an arrangement of $n+1$ hyperplanes of $\AA^n$ in general position with multiplicities,in the end proving Theorem \ref{exponentes}. But first, let us treat another case as a warm up. We will indeed give an alternative proof of a well-known fact regarding quasi-homogeneous singularities:

\begin{prop}\label{casi homogeneo}
Let $f\in\kk[x_1,\ldots,x_n]$ be a quasi-homogeneous polynomial of degree $d$ with respect to a system of integer weights $v=(v_1,\ldots,v_n)$ such that $\gcd(v_1,\ldots,v_n)=1$. Then, $\alpha\in\kk$ is an exponent of some cohomology of $f_+\OO_{\AA^n}$ at the origin only if $d\alpha\in\Z$.
\end{prop}
\begin{proof}
By virtue of Theorem \ref{Koszul}, we will prove the equivalent statement that for any $\alpha$ such that $d\alpha$ is not an integer, the $\kk$-linear homomorphism $\Phi: R^{n+1}\lra R$ given by
$$\Phi=\left(f-t,\dd_1+f'_1\fia,\ldots, \dd_n+f'_n\fia\right)$$
is surjective. Note that in this case, $g(\xx)=1$ and then $R=\kk((t))[\xx]$.

Let us pick an element $c$ of $R$, and let us say that there exist $a$, and $n$ polynomials $b^i$ for every $i=1,\ldots,n$, so that $\Phi(a,b^1,\ldots,b^n)=c$. To prove that we can assume, without loss of generality, that $a$, each of the $b^i$ and $c$ are quasi-homogeneous (say of $v$-degrees $m,m+1,\ldots,m+1,m$ for $m\geq0$), allowing them to vanish. We will have
\begin{equation}\label{formulaqh}
\left\{\begin{array}{rcl}
\displaystyle fa+\sum_{i=1}^nf'_i\fia b^i&=&0,\\
\displaystyle-ta+\sum_{i=1}^n\dd_ib^i&=&c.\end{array}\right.
\end{equation}
Here, $f$ is a quasi-homogeneous polynomial, so the Euler formula $df=\sum_iv_if'_i$ holds. That can be thought of as a syzygy of the Jacobian ideal $(f,f'_1,\ldots,f'_n)$, whose first term $-d$ is of $v$-degree zero. Then we can assume that any other syzygy of that ideal deals only with the partial derivatives of $f$. As a consequence, we know from the first equation that there exist quasi-homogeneous polynomials $F, g_{i,j}\in R$ for $i$ running in some finite set of indexes $I$ and $j=1,\ldots,n$, such that
\begin{equation}\label{quasi}
\begin{array}{l}
a=-dF,\\
\displaystyle\fia b^j=v_jx_jF+\sum_{i\in I}s_{i,j}g_{i,j}\,j=1,\ldots,n,\end{array}
\end{equation}
the $s_{i,j}$ being the components of the syzygy $s_i$ of $(f'_1,\ldots,f'_n)$.

In fact, we do not need so much generality to achieve our goal. Namely, we will assume in the following that all $g_{i,j}$ are zero, greatly simplifying expression \ref{quasi}. That assumption will not affect the validity of the statement, as we will see. Substituting the new values of $f$ and its partial derivatives in \ref{quasi} we get that
$$tdF+\sum_{i=1}^n\dd_i\fiai v_ix_i F=tdF+\fiai\left(\sum_{i=1}^n(v_ix_i\dd_i+v_i)F\right)=
dA_{\frac{m+|v|}{d}}F=c;$$
the first equality is just commuting $x_i$ and $\dd_i$ and a consequence of the fact that $\fia$ commutes with anything independent of $t$, whereas the second uses the notation introduced in Definition \ref{A-operators sencillos} and the Euler formula for $F$. Note that if the system of weights $(v_1,\ldots,v_n)$ were not reduced, the common factors of $m+|v|$ would have canceled themselves with the respective ones of $d$, and the final operator $A:=A_{\frac{m+|v|}{d}}$ would be exactly the same.

Now if $d\alpha$ is not an integer, $A$ is invertible, and so the morphism $\Phi$ is surjective; note that if we had not assumed that the $g_{(i,j)}$ vanish, the argument could have been the same. Thus $\alpha$ cannot be an exponent at the origin of any of the cohomologies of $f_+\OO_{\AA^n}$, as we wanted to prove.
\end{proof}

Let us continue now towards the proof of Theorem \ref{exponentes} and set some notation of use from now on. Recall that $(w_0,\ldots,w_n)\in\Z_{>0}^{n+1}$ is some $(n+1)$-tuple of positive integers. Under a suitable linear change of variables, we are able to work with the polynomial $\lambda=x_1^{w_1}\cdot\ldots\cdot x_n^{w_n}(1-x_1-\cdots-x_n)^{w_0}$. The order of the exponents $w_0,\ldots,w_n$ is irrelevant; we could just reorder the variables or change one of them by $1-x_1-\cdots-x_n$. We have already indicated that this case turns out to be interesting when we study the Gauss-Manin cohomology of a generalized Dwork family and related to some problems arising in mirror symmetry. For a bigger number of hyperplanes, the computations seem quite complex, and we think that the particular case of having $n+1$ of them is interesting enough to see the applicability of the main result to other problems.

On the other hand, working with fewer hyperplanes is quite easy: assume that for some $r\geq0$ and every $i=0,\ldots,r$ we have $w_i=0$. Abusing some of the notation we can call our morphism the same as the polynomial, that is, such that $\lambda=x_{r+1}^{w_{r+1}}\cdot\ldots\cdot x_n^{w_n}$. The notion of exponent is deeply local, so we could work with the image of $\lambda$ outside the origin, that is to say, restrict it from $\AA^r\times\gm^{n-r}$ to $\gm$, without altering the set of exponents. Therefore under those assumptions we have by the K\"{u}nneth formula (cf. \cite[Prop. 1.5.30]{HTT}) that
$$\lambda_+\OO_{\AA^r\times\gm^{n-r}}\cong \left(\bigoplus_{i_1=1}^{w_{r+1}}\cdots\bigoplus_{i_{n-r}=1}^{w_n} \KK_{i_1/w_{r+1}}\ast\cdots\ast\KK_{i_{n-r}/w_n}\right)\otimes\left( \bigoplus_{i=-r}^0\OO_{\gm}^{\binom{r}{-i}}[-i]\right),$$
where the operation $\ast$ is the multiplicative convolution of $\DD_{\gm}$-modules defined as in \cite[5.1.7.1]{KaESDE}.

Note that for any $\alpha$ and $\beta$ in $\kk$,
$$\KK_\alpha\ast\KK_\beta\cong\left\{\begin{array}{ll}
\KK_\alpha[1]\oplus\KK_\alpha[0] & \text{ if }\alpha\equiv\beta\!\mod\Z,\\
0 & \text{ otherwise,}\end{array}\right.$$
by virtue of \cite[Lems. 5.2.1, 6.3.4]{KaESDE}. Then by repeatedly applying the claim we can affirm that
$$\lambda_+\OO_{\AA^n}\cong\left(\bigoplus_{\alpha\in A}\KK_\alpha\right)\otimes\left(\bigoplus_{i=-n+1}^0 \OO_{\gm}^{\binom{n-1}{-i}}[-i]\right),$$
where $A$ is the set of rational numbers $\alpha\in(0,1]$ for which there exist $i_1,\ldots,i_{n-r}$ such that $\alpha=i_j/w_{r+j}$ for every $j$. Note that, by the following lemma, $A=\{1\}$ if and only if $\gcd(w_{r+1},\ldots,w_n)=1$.

\begin{lema}
Let $n>1$ be an integer, and let $(w_1,\ldots,w_n)$ be an $n$-tuple of positive integers. The following conditions are equivalent:
\begin{enumerate}[1)]
\item There exists another $n$-tuple $(a_1,\ldots,a_n)$ of positive integers such that $a_i<w_i$ for every $i=1,\ldots,n$, and the quotients $a_i/w_i$ are all equal.
\item $\gcd(w_1,\ldots,w_n)>1$.
\end{enumerate}
\end{lema}
\begin{proof}
The upwards part of the equivalence is easy; the other implication can be proved by contradiction. Indeed, assume that $\gcd(w_1,\ldots,w_n)=1$ and apply B\'{e}zout's identity to obtain some integers $c_1,\ldots,c_n$ such that $1=\sum_ic_iw_i$. But then, calling $q=a_i/w_i$ for any $i=1,\ldots,n$, $1>q=\sum_ic_iqw_i=\sum_ic_ia_i$, which cannot (as it should) be a positive integer. Thus $\gcd(w_1,\ldots,w_n)>1$ and we are done.
\end{proof}

Let us deal then with the case of $n+1$ hyperplanes. However, note that we will not calculate the exponents of $\lambda_+\OO_{\AA^n}$, but stay with just a finite set of rational numbers as candidates (see Remark \ref{notacalculo} after the proof).

\noindent\textit{Proof of Theorem \ref{exponentes}.} As with the previous proposition, we will prove the equivalent statement that for any $\alpha$ such that $w_i\alpha$ is not an integer for any $i=0,\ldots,n$, the $\kk$-linear homomorphism $\Phi: R^{n+1}\lra R$ given by
$$\Phi=\left(\lambda-t,\dd_1+\lambda'_1\fia,\ldots, \dd_n+\lambda'_n\fia\right)$$
is surjective. Note that again in this case, $g(\xx)=1$ and $R=\kk((t))[\xx]$.

Let us assume that $n\geq2$, but we will comment throughout the proof on the changes needed to treat the case $n=1$.

For the sake of simplicity, let us denote by $\sigma$ and $d$, respectively, the sums $x_1+\cdots+x_n$ and $\sum w_i$. In the following, $l_i$ will mean $w_i\sigma+w_0x_i$ for each $i=1,\ldots,n$. Therefore, $\lambda'_i=\xx^{w-e_i}(1-\sigma)^{w_0-1}(w_i-l_i)$ for every $i$.

Exactly as we did in the previous proposition, let us pick an element $c$ of $R$, that we can assume as before without loss of generality to be homogeneous of degree $m\geq0$, and let us say that there exist $a$, and $n$ polynomials $b^i$ for every $i=1,\ldots,n$, so that $\Phi(a,b^1,\ldots,b^n)=c$, and see which conditions we have to impose on them. We will express the unknown polynomials $a$ and the $b^i$ in terms of some others and deduce some conditions on the new ones, although here those conditions will not be as simple as in the previous example. Then we will give a system of equations such that the existence of solutions to it implies the existence of $a$ and the $b^i$. In the end we will show how to find a solution of such a system using that none of $w_0\alpha$,\ldots,$w_n\alpha$ is an integer.

So let us return to our $a$, $b^i$ and $c$, such that $\Phi(a,b^1,\ldots,b^n)=c$. For every $r\geq0$, we will have in general that
\begin{equation}\label{formulageneral}
\sum_{j+k=r}\lambda_j a_k-ta_r+\sum_{i=1}^n\sum_{j+k=r}(\lambda'_i)_j\fia b_k^i+\sum_{i=1}^n\dd_ib_{r+1}^i=c_r.
\end{equation}
We will also assume that $a$ has only nonvanishing $k$th homogeneous components for $k=m,\ldots,m+d-1$, and each of the $b^i$ for $k=m+1,\ldots,m+d$. Thus our general formula \eqref{formulageneral} will be useful for us only for $r=m,\ldots,m+2d-1$. We will also assume $\alpha$ to be noninteger in order to be able to invert $\fia$ in the following.

Let us focus first on the expression
\begin{equation}\label{formulaalta}
\lambda a+\sum_i\lambda'_i\fia b^i=0
\end{equation}
that holds for degrees between $m+d$ and $m+2d-1$. From this fact we will obtain some additional, useful information about $a$ and the $b^i$.

If we take the factors common to every summand of formula \eqref{formulaalta}, we get
\begin{equation}\label{sicigia}
(1-\sigma)^{w_0-1}\left(a\xx^w(1-\sigma)+\sum_{i=1}^n\fia b^i\xx^{w-e_i} (w_i-l_i)\right)=0,
\end{equation}
so $(a(1-\sigma),\fia b^1(w_1-l_1),\ldots,\fia b^n(w_n-l_n))$ is a syzygy of the sequence consisting of the monomial $\xx^w$ and its $n$ partial derivatives. Therefore, since it forms a monomial ideal, $x_i$ divides $\fia b^i$ for every $i$. Let us write $\fia b^i=x_i\bb^i$, so that we can divide by $\xx^w$ in formula \eqref{sicigia} to obtain
\begin{equation}\label{sicigiasimple}
a(1-\sigma)+\sum_i\bb^i(w_i-l_i)=0,
\end{equation}
which, recall, will be valid only for degrees from $m+1$ to $m+d$.

Let us start then with formula \eqref{sicigiasimple} by degree $m+d$. We have
$$a_{m+d-1}\sigma+\sum_{i=1}^nl_i \bb_{m+d-1}^i= \sum_{i=1}^n\left(a_{m+d-1}+\sum_{j=1}^nw_j\bb_{m+d-1}^j+w_0 \bb_{m+d-1}^i\right)x_i=0.$$
We could argue that the $x_1,\ldots,x_n$ form a regular sequence in order to obtain an expression for their ``coefficients'' in terms of other polynomials.
Nevertheless, as we did in the proof of the previous Proposition, we will make some assumptions to simplify our calculations. Namely, we will assume that every sum $a_{m+d-1}+\sum_jw_j\bb_{m+d-1}^j+w_0 \bb_{m+d-1}^i$ vanishes. Moreover, we will also assume that all of the $\bb_{m+d-1}^i$ are equal. However, as we will explain and see alongside the proof, all these suppositions and the subsequent ones will not prevent us from proving the theorem.

Let us now rename $\bb_{m+d-1}^i=f^0$, for every $i$; $f^0$ is a homogeneous polynomial of degree $m+d-1$. In the end we can also write that $a_{m+d-1}=-df^0$.

Let us go on by taking $r=m+d-1$. Our equation \eqref{sicigiasimple} turns into
$$a_{m+d-2}\sigma-a_{m+d-1}+\sum_{i=1}^nl_i \bb_{m+d-2}^i-\sum_{i=1}^nw_i \bb_{m+d-1}^i=0.$$
We can replace $a_{m+d-1}$ and the $\bb_{m+d-1}^i$ by their values in terms of $f^0$, and get
\begin{equation}\label{fcero}
a_{m+d-2}\sigma+\sum_{i=1}^nl_i \bb_{m+d-2}^i+w_0f^0=0.
\end{equation}
Note that, since $f^0$ is homogeneous of degree $m+d-1>0$, there exist $n$ homogeneous polynomials $f_{(1)}^0,\ldots,f_{(n)}^0\in R$ of degree $m+d-2$ such that $f^0=\sum_ix_if_{(i)}^0$. Replace $f^0$ by that sum in the formula above. In addition, as before, assume that every factor of $x_i$ in formula \ref{fcero} is zero for $i=1,\ldots,n$ and that all of the sums $\bb_{m+d-2}^i+f_{(i)}^0$ are equal to some new homogeneous polynomial in $R$, named $f^{1}$, of degree $m+d-2$ (note that if $n=w_0=w_1=1$ and $m=0$, it would be constant, stopping this process here). Finally we have
$$\begin{array}{l}
\displaystyle a_{m+d-2}=-df^{1}+\sum_{j=1}^nw_jf_{(j)}^0\\
\displaystyle \bb_{m+d-2}^i=f^{1}-f_{(i)}^0,\,i=1,\ldots,n.\end{array}$$
Let us move on and see what happens with formula \eqref{sicigiasimple} when the degree is $m+d-2$. Our favorite formula reads
$$a_{m+d-3}\sigma-a_{m+d-2}+\sum_{i=1}^nl_i \bb_{m+d-3}^i-\sum_{i=1}^nw_i \bb_{m+d-2}^i=0.$$
Writing $a_{m+d-2}$ and the $\bb_{m+d-2}^i$ as with higher degrees and proceeding like with degree $m+d-1$ yields that the terms in the $f_{(i)}^0$ vanish, so we can proceed exactly as in the previous step.

More concretely, taking lower and lower degrees in \eqref{sicigiasimple} as long as it is possible and renaming the subsequent $f_{(j)}^{k}$ that appear, we finally get
\begin{equation}\label{formulasdeayb}
\begin{array}{l}
\displaystyle a=\sum_{i=1}^n(-dx_i+w_i)F_{(i)}-d\tilde{F}\\
\displaystyle\fia b^i=x_i\left(\sum_{j=1}^nx_jF_{(j)}-F_{(i)}+\tilde{F}\right) ,\,i=1,\ldots,n,\end{array}
\end{equation}
where the $F_{(i)}$ are polynomials of $R$ that have only nonvanishing $k$th homogeneous components for $k=m,\ldots,m+d-2$, and $\tilde{F}$ is a homogeneous polynomial of degree $m$. In other words, each of the $f_{(j)}^{k}$ is now the homogeneous component of degree $m+d-2-k$ of $F_{(j)}$, for $j=1,\ldots,n$ and $k=0,\ldots,d-2$, and $\tilde{F}$ is just $f^{d-1}$.

Summing up, we have been able to express our first unknowns, the polynomials $a$ and $b^i$, in terms of many other polynomials, and we do not know anything about them but their degrees. However, recall that we still have $d$ equations left arising from our general formula \eqref{formulageneral} in degrees $m$ to $m+d-1$. These are the ones that will give us some information about our new unknowns, and are, in reverse order of degree (note that $\lambda$ lies in degrees $d-w_0$ to $d$),
\begin{equation}\label{formulabaja}
\left\{\begin{array}{l}
\displaystyle\sum_{j+k=r}\lambda_j a_k+\displaystyle\sum_{i=1}^n\displaystyle\sum_{j+k=r}(\lambda'_i)_j\fia b_k^i\\
-ta_r+\displaystyle\sum_{i=1}^n\dd_ib_{r+1}^i=0,\,r=m+d-w_0,\ldots,m+d-1,\\
-ta_r+\displaystyle\sum_{i=1}^n\dd_ib_{r+1}^i=0,\,r=m+1,\ldots,m+d-w_0-1,\\
-ta_m+\displaystyle\sum_{i=1}^n\dd_i b_{m+1}^i=c_m.
\end{array}\right.
\end{equation}
Let us find an expression for the system above. Recall the new expressions for $a$ and $b$ in \eqref{formulasdeayb} and take into account those of $\lambda$ and its partial derivatives:
\begin{align*}
\displaystyle\lambda_{d-w_0+j}=&\binom{w_0}{j}(-1)^j\xx^w\sigma^j,\\
\lambda'_{d-w_0+j-1,i}=&\left\{\begin{array}{ll}
w_i\xx^{w-e_i}, & \,j=0,\\
\displaystyle(-1)^j\xx^{w-e_i}\left(\binom{w_0}{j}\sigma^jw_i+ \binom{w_0-1}{j-1}w_0\sigma^{j-1}x_i\right), & \,j=1,\ldots,w_0.\end{array}\right.\end{align*}
Now, if we put all that into the remaining equations in \eqref{formulabaja}, we obtain the system
{\footnotesize\begin{equation}\label{sistema}
\left\{\begin{array}{l}
\displaystyle\sum_{i=1}^ndx_iA_{\frac{m+d+n-1}{d}}F_{(i),m+d-2}+(-1)^{w_0}w_0 \xx^w\sigma^{w_0-1}\tilde{F}=0,\\
\vdots\\
\displaystyle\sum_{i=1}^ndx_iA_{\frac{m+d+n-r}{d}}F_{(i),m+d-r-1}- \sum_{i=1}^nw_iA_{\frac{D_{i,1}}{w_i}}F_{(i),m+d-r}\\
\displaystyle+(-1)^{w_0-r+1}w_0 \binom{w_0-1}{w_0-r}\xx^w\sigma^{w_0-r}\tilde{F}=0,\, r=2,\ldots,w_0-1,\\
\vdots\\
\displaystyle\sum_{i=1}^ndx_iA_{\frac{m+d+n-w_0}{d}}F_{(i),m+d-w_0-1}- \sum_{i=1}^nw_iA_{\frac{D_{i,1}}{w_i}}F_{(i),m+d-w_0}-w_0\xx^w\tilde{F}=0,\\
\vdots\\
\displaystyle\sum_{i=1}^ndx_iA_{\frac{m+n+r}{d}}F_{(i),m+r-1} -\sum_{i=1}^nw_iA_{\frac{D_{i,1}}{w_i}}F_{(i),m+r}=0,\,r=1,\ldots,d-w_0-1,\\
\vdots\\
dA_{\frac{m+n}{d}}\tilde{F} -\displaystyle\sum_{i=1}^nw_iA_{\frac{D_{i,1}}{w_i}}F_{(i),m}=c_m.
\end{array}\right.
\end{equation}}
Note that, as in the proof of Proposition \ref{casi homogeneo}, the operators $A_\beta$ and $A_{\frac{D_i,1}{w_i}}$ are obtained from the summand $-ta_r$ in every homogeneous equation of (\ref{formulabaja}), together with $\fiai\sum_i\dd_ix_i$ of some homogeneous polynomial. Applying Euler's formula allows us to get rid of the sum of the derivatives. Let us try now to prove that this system has a solution, and then, that $\Phi$ is surjective.

Let us denote by $S_k$ the set $\{\underline{u}\in\mathbb{N}^n\,:\,|\underline{u}|=k\}$. We will say that the support of a homogeneous polynomial $P$ of degree $k$ is maximal if it is the whole $S_k$. Write $\bar{F}$ for $\sum_ix_iF_{(i)}$. Note that this polynomial comes from the $f^k$ obtained in the first part of the proof. For each $k=1,\ldots,m+d-1$ we obviously cannot have a priori that $\supp(\bar{F}_{m+k})\neq S_{m+k}$. Then we could choose the support of all of the $F_{(i),m+k}$, up to a reordering on the set of monomials that appear in each one. For instance, suppose $n=2$ and $m+k=2$. Then $\bar{F}_2=F_{20}x^2+F_{11}xy+F_{02}y^2$, and we could take either $F_{(1),1}=F_{20}x+F_{11}y$ and $F_{(2),1}=F_{22}y$ or $F_{(1),1}=F_{20}x$ and $F_{(2),1}=F_{11}x+F_{22}y$.

As a consequence, without loss of generality, we can, and will, assume the maximality of the supports of the polynomials
$$F_{(1),m+k}\text{ for }k=0,\ldots,w_1-1,$$
and for every $i=2,\ldots,n$,
$$F_{(i),m+k}\text{ for }k=w_1+\cdots+w_{i-1},\ldots,w_1+\cdots+w_{i}-1.$$
(Obviously this definition of maximality and the assumptions on the $F_{(i),m+k}$ are useless when $n=1$.) Thanks to the choice of $\alpha$ and Remark \ref{Adifiso} we know that each $A_{\frac{D_{i,1}}{w_i}}$ is invertible, so we can solve any $F_{(i),m+r}$ of maximal support in terms of $F_{(i),m+r-1}$ and $\tilde{F}$, for $r=0,\ldots,d-w_0-1$, over all the possible support of the corresponding equation.

Now is when the choice of the supports of the $F_{(i),m+k}$ makes sense. Start at the last equation of \eqref{sistema} by solving $F_{(1),m}$ and replace its value in the preceding equation, and do this with the polynomial $F_{(i),m+k}$ having a maximal support until we reach the $w_0$th equation. Assume that every unused polynomial $F_{(i),m+k}$ vanishes (again this assumption does not endanger the generality of the proof). As a consequence of all that, we reduce ourselves to dealing with a newer system of only $w_0$ equations, consisting of the first $w_0-1$ equations of the preceding system and a new $w_0$th equation, namely
$$\xx^w\left(\Upsilon A_{\frac{m+n}{d}}-w_0\right)\tilde{F}- \sum_{i=1}^nw_iA_{\frac{D_{i,1}}{w_i}}F_{(i),m+d-w_0}=\xx^w\Upsilon c_m,$$
where, by Lemma \ref{Adifcom},
\begin{align*}
  \Upsilon= & d^d\prod_{i=1}^n  w_i^{-w_i}A_{\frac{m+d+n-1}{d}} A_{\frac{D_{n,w_n}}{w_n}}^{-1}\cdot\ldots\cdot A_{\frac{m+d+n-w_n}{d}}A_{\frac{D_{n,1}}{w_n}}^{-1}\cdot\ldots\cdot A_{\frac{m+n+1}{d}}A_{\frac{D_{1,1}}{w_1}}^{-1}\\
  = & \prod_{k=1}^{d-1}A_{\frac{m+d+n-k}{d},k-1}\left(\prod_{i=1}^n \prod_{j=1}^{w_i}A_{\frac{D_{i,j}}{w_i}, j+w_1+\cdots+w_{i-1}-1}\right)^{-1}.
\end{align*}

Let us simplify the system once more; as before, although we lose some generality, this new assumption will not only preserve enough of it, but leave the equations in a more manageable form. More concretely, assume that, for every $r=2,\ldots,w_0$, the polynomials $F_{(i),m+d-r}$ coincide for every $i=1,\ldots,n$ and are divisible by $\xx^w\sigma^{w_0-r}$. Write $F_{(i),m+d-r}=\xx^w\sigma^{w_0-r}F_{m+d-r}$ for all those values of $i$ and $r$ (note that every polynomial $F_{m+d-r}$ is homogeneous of degree $m$). Thanks to that hypothesis, we can divide by $\xx^w\sigma^{w_0-r+1}$ each corresponding equation to get the simpler system of homogeneous polynomials of degree $m$
\begin{equation}\label{sistemacorto}
\left\{\begin{array}{l}
dA_{\frac{m+d+n-1}{d}}F_{m+d-2}+(-1)^{w_0}w_0\tilde{F}=0,\\
\vdots\\
dA_{\frac{m+d+n-r}{d}}F_{m+d-r-1}- (d-w_0)A_{\frac{m+d+n-r}{d-w_0}}F_{m+d-r}\\
\displaystyle +(-1)^{w_0-r+1}w_0 \binom{w_0-1}{w_0-r}\tilde{F}=0,\, r=2,\ldots,w_0-1,\\
\vdots\\
\left(\Upsilon A_{\frac{m+n}{d}}-w_0\right)\tilde{F}- (d-w_0)A_{\frac{m+d+n-w_0}{d-w_0}}F_{m+d-w_0}=\Upsilon c_m. \end{array}\right.
\end{equation}
This system is the one that we will prove to have a solution, so that we finally show that $\Phi$ is surjective.

The $A_\beta$ do not commute pairwise, so in principle we cannot deal with the determinant of the matrix of the system. However, under an easy change of variables, we can see the $A_\beta$ as elements of a commutative subring of the ring of endomorphisms of $R$. By our assumption on $\alpha$, the endomorphism $D_\alpha$ of $\kk((t))$ is invertible, so we can define a new operator $B_\beta$ as $A_\beta\alpha t^{-1}=\alpha(1+\beta D_\alpha^{-1})$, which is an element of $\kk[D_\alpha^{-1}]$, a commutative ring whose action on $\kk((t))$ is defined by $D_\alpha^{-1}t^l=(l+\alpha)^{-1}t^l$. Now $\alpha t^{-1}$ is an isomorphism of $R$, so we can rename the $F_k$ to mean $t\alpha^{-1}F_k$, for each $k=m+d-w_0,\ldots,m+d-2$.

Now every coefficient of system \eqref{sistemacorto} is of the form of some $B_\beta$ and thus lives in $\kk[D_\alpha^{-1}]$, except for $\Upsilon B_{\frac{m+n}{d}}$ in the final equation, which has degree 1 in $t$. Nevertheless, note that this operator goes together with $-w_0$, so by Lemma \ref{deforgrado} its sum is an automorphism of $R$. Moreover, regarding just the existence of solutions to the system and not their actual form, we can restrict ourselves to working only with $-w_0$.

If $w_0=1$, then we have only a single equation, from which we can solve $\tilde{F}$ and thus the system, showing the existence of solutions. In the following we will assume that $w_0\geq2$.

We have finally arrived at a point where we have a matrix of coefficients in $\kk[D_\alpha^{-1}]$, so we just need to show that its determinant is an invertible endomorphism of $\kk((t))$. If we manage to do so, we will have proved that system \eqref{sistemacorto} has a solution, but that implies the existence of a solution to system \eqref{sistema} and this, in turn, implies the existence of a preimage to our $c$.

Expanding it along the last column, the determinant is
{\scriptsize$$\left|\begin{array}{ccccc}
dB_{\frac{m+d+n-1}{d}} & 0 & \cdots & 0 & (-1)^{w_0}w_0\\
-(d-w_0)B_{\frac{m+d+n-2}{d-w_0}} & dB_{\frac{m+d+n-2}{d}} & \cdots & 0 & (-1)^{w_0-1}w_0(w_0-1)\\
0 & -(d-w_0)B_{\frac{m+d+n-3}{d-w_0}} & \ddots & \vdots & \vdots\\
\vdots & \vdots & \ddots & dB_{\frac{m+d+n-w_0+1}{d}} & w_0(w_0-1)\\
0 & 0 & \cdots & -(d-w_0)B_{\frac{m+d+n-w_0}{d-w_0}} & -w_0
\end{array}\right|=$$}
$$=-w_0\sum_{r=1}^{w_0}\binom{w_0-1}{w_0-r}(-1)^{w_0-r} \prod_{k=1}^{r-1}d^{r-1}B_{\frac{m+d+n-k}{d}} \prod_{k=r}^{w_0-1}(d-w_0)^{w_0-r}B_{\frac{m+d+n-k-1}{d-w_0}}.$$
Call the determinant above $-w_0\Delta_\alpha$. Recall that, up to now, we have used that $w_i\alpha$ is not an integer for any $i=1,\ldots,n$. Now is when we will use that $w_0\alpha$ is not an integer either. Since every operator $B_\beta$ preserves powers of $t$ up to a coefficient in $\kk$, so does $\Delta_\alpha$; say $\Delta_\alpha\left(t^l\right)=d_{\alpha,l}t^l$. But then there will be some power of $t$ in its kernel if and only if $\Delta_\alpha$ is not bijective. Therefore, we need to show just that, for the values of $\alpha$ under consideration, $d_{\alpha,l}$ does not vanish for every $l\in\Z$. It is easy to see that $d_{\alpha,l}$ is
{\scriptsize $$\sum_{r=1}^{w_0}\binom{w_0-1}{w_0-r}(-1)^{w_0-r} \prod_{k=1}^{r-1}\left(d\alpha+\alpha\frac{m+d+n-k}{l+\alpha}\right) \prod_{k=r}^{w_0-1}\left((d-w_0)\alpha+\alpha\frac{m+d+n-k-1}{l+\alpha} \right)$$}
{\footnotesize$$=q_{\alpha,l}^{w_0-1} \sum_{r=1}^{w_0}\binom{w_0-1}{w_0-r}(-1)^{w_0-r} \prod_{k=1}^{r-1}(d(\alpha+l)+m+d+n-k) \prod_{k=r}^{w_0-1}((d-w_0)(\alpha+l)+m+d+n-k-1),$$}
where $q_{\alpha,l}$ is the quotient $\alpha/(l+\alpha)$. Up to the factor $q_{\alpha,l}^{w_0-1}$, the expression above is a polynomial in $\alpha$ of degree $w_0-1$, so there will be at most $w_0-1$ values of $\alpha$ so that it vanishes. In fact, for a fixed $l$, they are $-l-a/w_0$, for $a=1,\ldots,w_0-1$.

Indeed, let $a$ be as above. Then $q_{-l-a/w_0,l}^{-(w_0-1)}d_{-l-a/w_0,l}$ can be written as
{\footnotesize$$\prod_{k=1}^{w_0-a}\left(-\frac{d}{w_0}a+m+d+n-k\right) \sum_{r=1}^{w_0}\binom{w_0-1}{w_0-r}(-1)^{w_0-r} \prod_{j=1}^{a-1}\left(-\frac{d}{w_0}a+m+n+d-r+j\right)$$}
$$=C_a\sum_{r=1}^{w_0}\binom{w_0-1}{w_0-r}(-1)^{w_0-r}p_a(r),$$
where $C_a$ and $p_a$ are, respectively, a constant and a polynomial of degree $a-1\leq w_0-2$. Now thanks to the following lemma, we can deduce the vanishing of the determinant, so we have finally found that for every $\alpha$ such that $w_i\alpha$ is not an integer, the original system has a solution.

All of this process could be made independently of the choice of $m$, $n$, all of the $w_i$ and $c_m$, so it finally proves the surjectivity of $\Phi$.

The claim about the multiplicity of the exponents follows from two facts: the order of the $w_i$ does not play any role in the complex $\lambda_+\OO_{\AA^n}$ as we commented in the introduction to the context of the proposition, and within each set of possible exponents $\left\{1/w_i,\ldots,(w_i-1)/w_i\right\}$, the proof shows that whichever we use, the outcome is the same because of the different possible values of $m$.
\hfill$\square$

\begin{nota}\label{notacalculo}
Note that we already knew before the proposition that by the monodromy theorem, the exponents were to be rational. What this result provides, with respect to that fact, is a much shorter set of possible exponents.

The converse of the last proposition holds in a stronger way: every cohomology of $\lambda_+\OO_{\AA^n}$ is constant, except the last one, $\HH^0\lambda_+\OO_{\AA^n}$, whose exponents are exactly those of the form $j/w_i$, for $j=1,\ldots,w_i$ and $i=0,\ldots,n$. However, it needs much more preparatory work; it can be found in \cite[\S5]{CaDwork}.
\end{nota}

\begin{lema}
Let $m$, $n$ be two integers such that $0\leq m<n-1$. Then,
$$\sum_{k=1}^n\binom{n-1}{n-k}(-1)^{n-k}k^m=0.$$
\end{lema}
\begin{proof}
First of all rewrite the formula above as
$$\sum_{k=0}^{n-1}\binom{n-1}{k}(-1)^k(n-k)^m,$$
which is obviously true if $n=2$, so assume $n>2$. We will be done as long as we can show that $\sum_{k=1}^{n-1}\binom{n-1}{k}(-1)^kk^m=0$ for every $n$ and $m<n-1$. Let us do it by induction on $n$ and $m$. If $n=3$ or $m=0$ it is also very easy to prove it.

Let us go therefore for a general $n$, take some $n-1>m>0$ and assume the validity of $\sum_{k=1}^{a-1}\binom{a-1}{k}(-1)^kk^b=0$ for every $a<n$ and $b<\min(m,a-1)$. Now note that
$$\sum_{k=1}^{n-1}\binom{n-1}{k}(-1)^kk^m=(-1)^{n-1}(n-1)^m+ \sum_{k=1}^{n-2} \left(\binom{n-2}{k-1}+\binom{n-2}{k}\right)(-1)^kk^m$$
$$=\sum_{k=1}^{n-2}\binom{n-2}{k}(-1)^k\left(k^m- (k+1)^m\right).$$
Since $k^m-(k+1)^m$ is a polynomial of degree $m-1$ in $k$, we just need to apply the induction hypothesis to finish.
\end{proof}

\end{document}